\documentclass[12pt]{article}
\usepackage[T1]{fontenc}
\usepackage{amsmath,amsthm,stmaryrd}
\usepackage[utopia]{mathdesign}
\usepackage[all]{xy}
\usepackage{hyperref}

\newcommand{\mr}[1]{\mathrm{#1}}
\newcommand{\mf}[1]{\mathfrak{#1}}
\newcommand{\mc}[1]{\mathcal{#1}}
\newcommand{\mb}[1]{\mathbb{#1}}

\newcommand{\Z}{\mb{Z}}
\newcommand{\Q}{\mb{Q}}
\newcommand{\zp}{\mb{Z}_p}
\newcommand{\qp}{\mb{Q}_p}
\newcommand{\C}{\mb{C}}

\newcommand{\F}{\mb{F}}

\newcommand{\cotimes}[1]{\,\hat{\otimes}_{#1} \,}
\newcommand{\cozp}{\cotimes{\zp}}

\newcommand{\La}{\Lambda}
\newcommand{\Iw}{\mr{Iw}}
\newcommand{\ord}{\mr{ord}}
\newcommand{\qbar}{\overline{\Q}}

\newcommand{\T}{\mathcal{T}}
\newcommand{\tT}{\tilde{\mathcal{T}}}

\newcommand{\invlim}{\varprojlim}

\newcommand{\ps}[1]{\llbracket #1 \rrbracket}

\DeclareMathOperator{\Hom}{Hom}
\DeclareMathOperator{\End}{End}
\DeclareMathOperator{\Gal}{Gal}

\DeclareMathOperator{\SL}{SL}

\DeclareMathOperator{\Cor}{Cor}

\newtheorem{theorem}{Theorem}[subsection]
\newtheorem{proposition}[theorem]{Proposition}
\newtheorem{lemma}[theorem]{Lemma}
\newtheorem{corollary}[theorem]{Corollary}

\theoremstyle{definition}
\newtheorem{definition}[theorem]{Definition}

\theoremstyle{remark}
\newtheorem{remark}[theorem]{Remark}

\newtheorem*{ack}{Acknowledgments}

\newcounter{counti}
\newenvironment{enum}{\begin{list}{\alph{counti}.}{\usecounter{counti}
\setlength{\topsep}{3pt}
\setlength{\labelwidth}{0bp}
\setlength{\leftmargin}{0pt}
\setlength{\itemindent}{23pt}
\setlength{\itemsep}{5pt}
\setlength{\parsep}{0pt}}
}{\end{list}} 

\newcounter{countii}

\renewcommand{\baselinestretch}{1.2}
\numberwithin{equation}{section}

\begin{document}

\title{Modular symbols and the integrality of zeta elements}
\author{Takako Fukaya, Kazuya Kato, Romyar Sharifi}
\date{}
\maketitle

\begin{center}\vspace{-3ex}
Dedicated to Professor Glenn Stevens on the occasion of his 60th birthday 
\end{center}

\begin{abstract}
	We consider modifications of Manin symbols in first homology groups of modular curves with $\zp$-coefficients for an odd prime $p$. We show that these symbols generate homology in primitive eigenspaces for the action of diamond operators, with a certain condition on the eigenspace that can be removed on Eisenstein parts. We apply this to prove the integrality of maps taking compatible systems of Manin symbols to compatible systems of zeta elements. In the work of the first two authors on an Iwasawa-theoretic conjecture of the third author, these maps are constructed with certain bounded denominators. As a consequence, their main result on the conjecture was proven after inverting $p$, and the results of this paper allow one to remove this condition.
\end{abstract}

\section{Introduction}

\subsection{Homology and $(c,d)$-symbols}

Let $p \ge 5$ be a prime, and let $N$ be a positive integer.  We let $X_1(N)$ denote the compact modular curve of level $N$ over $\C$ and let $C_1(N)$ denote its set of cusps.  Consider the relative homology group
$$
	\tilde{H} = H_1(X_1(N),C_1(N),\zp).
$$
Manin showed that $\tilde{H}$ is generated by elements $[u:v]$ attached to pairs $(u,v)$ with 
$u, v \in \Z/N\Z$ and $(u,v) = (1)$ (see Section \ref{homology}).
For integers $c$ and $d$ greater than $1$ and prime to $6N$, we define the
$(c,d)$-symbol for the pair $(u,v)$ by
$$
	{}_{c,d}[u:v] = c^2d^2[u:v] - c^2[u:dv] - d^2[cu:v] + [cu:dv].
$$
Let $\tilde{C}$ denote the $\zp$-span in $\tilde{H}$ of all $(c,d)$-symbols as we vary $c$ and $d$.  We will examine the difference between $\tilde{C}$ and $\tilde{H}$.

Now, let $M$ a positive integer such that $p \nmid M\varphi(M)$, and set $N = M$ or $Mp$.
Let $\theta$ be an even $p$-adic character of $(\Z/N\Z)^{\times}$ of conductor divisible by $M$.
Let $\tilde{H}^{(\theta)}$
denote the $\theta$-eigenspace of $\tilde{H}$ under the action of inverse diamond operators, and similarly for $\tilde{C}$.
Let $\mc{O}_{\theta}$ denote the $\zp$-algebra generated by the values of $\theta$.
Let  $\omega$ denote the Teichm\"uller character on $(\Z/p\Z)^{\times}$, 
and note that we can view both it and $\theta$ as characters on $(\Z/Mp\Z)^{\times}$.  
We prove the following result (and a counterpart for $p = 3$):

\begin{theorem} \label{mainthm}
	If $\theta \neq \omega^2$, then the $\mc{O}_{\theta}$-module $\tilde{H}^{(\theta)}$ is spanned by 
	$\tilde{C}^{(\theta)}$ and 
	the projection of $[1:p]$ to $\tilde{H}^{(\theta)}$.
	Moreover, if $\theta\omega^{-2}|_{(\Z/p\Z)^{\times}} \neq 1$, then $\tilde{H}^{(\theta)} = \tilde{C}^{(\theta)}$.
\end{theorem}

Let $I$ denote the Eisenstein ideal of the Hecke algebra $\mf{H}$ in $\End_{\zp}(\tilde{H})$ generated by $T_{\ell} - 1 - \ell\langle \ell \rangle$ for primes $\ell \nmid N$ and $U_{\ell} - 1$ for primes $\ell \mid N$, and let $\mf{m} = I + (p,\langle 1+p \rangle - 1)$.  We exhibit the following strengthening of Theorem \ref{mainthm} on Eisenstein components (i.e., for the localization $\tilde{H}^{(\theta)}_{\mf{m}}$ of $\tilde{H}^{(\theta)}$ at the maximal ideal $\mf{m}\mf{H}^{(\theta)}$ of $\mf{H}^{(\theta)}$).

\begin{theorem} \label{eisthm}
	For $\theta \neq \omega^2$, 
	the $\mc{O}_{\theta}$-modules $\tilde{H}^{(\theta)}_{\mf{m}}$ and $\tilde{C}^{(\theta)}_{\mf{m}}$ are equal.
\end{theorem}

Theorems \ref{mainthm} and \ref{eisthm} are derived in Section \ref{homology} from the results
of Section \ref{generate}.  We actually derive these results as simple corollaries of variants of arithmetic interest for the $\zp$-submodule $H$ of $\tilde{H}$ generated by those Manin symbols $[u:v]$ with $u, v \neq 0$.
This $H$ is the relative homology group $H_1(X_1(N),C_1^0(N),\zp)$, where $C_1^0(N)$ denotes those cusps
of $X_1(N)$ that do not lie above the zero cusp of $X_0(N)$.  

\begin{remark} \label{mainremark}
	If we let $C$ denote the $\zp$-span of all $(c,d)$-symbols ${}_{c,d}[u:v]$ for nonzero $u$ and $v$, then 
	both of the above results hold with $\tilde{H}$ replaced by $H$ and $\tilde{C}$ replaced by $C$, the second 
	theorem even for $\theta = \omega^2$.
\end{remark}

In Section \ref{homology}, we explain how these results can be used to prove that compatible sequences
of $(c,d)$-symbols generate the ordinary and Eisenstein parts of homology up the tower $X_1(Mp^r)$ of 
modular curves.

\subsection{Application to Iwasawa theory}

Let $\mc{S}$ denote the space of $\La$-adic cusp forms of level $Mp^{\infty}$, where $\La \cong \zp\ps{1+p\zp}$
is the usual Iwasawa algebra.  The group elements in $\La$ act as inverses of diamond operators on $\mc{S}$.  Let $I$ denote the Eisenstein ideal in the Hecke algebra acting on $\End_{\zp}(\mc{S})$, defined as above.  
Let $X$ denote the unramified Iwasawa module over the cyclotomic $\zp$-extension $K$ of $\Q(\mu_N)$, with
$N = Mp$, i.e., the Galois group of the maximal unramified abelian pro-$p$ extension of $K$.  
It has the structure of a $\zp\ps{\Gal(K/\Q)} \cong \La[(\Z/N\Z)^{\times}]$-module under conjugation by lifts of group elements. 

Suppose now that $\theta\omega^{-1}$ has conductor $N$ or that $\theta\omega^{-1}|_{(\Z/M\Z)^{\times}}(p) \neq 1$, and also that $\theta \neq 1, \omega^2$. 
Set 
\begin{eqnarray*}
	P = \mc{S}^{(\theta)}/I\mc{S}^{(\theta)} &\mr{and}& Y = X(1)^{(\theta)}
\end{eqnarray*}
As defined in \cite{sh-Lfn, fk-proof}, there are canonical $\La$-module homomorphisms
\begin{eqnarray*}
	\varpi \colon P \to Y	&\mr{and}& \Upsilon \colon Y \to P,
\end{eqnarray*}
the first being explicitly defined to take a sequence of Manin symbols to a certain sequence of cup products
of cyclotomic $N$-units (see also \cite{busuioc} at level $p$), and the second being defined via the action of Galois on \'etale cohomology groups of modular curves (and constructed out of maps occurring in Hida-theoretic proofs of the main conjecture \cite{wiles,ohta-ord}).  In \cite{sh-Lfn}, it was conjectured that these maps are inverse to each other, up to a unit suspected to be $1$.   This was shown to be true in \cite{fk-proof} upon tensoring
with $\qp$, after multiplication by a power
series $\xi'$ interpolating the derivative of the $p$-adic $L$-function of $\omega^2\theta^{-1}$.  
That is, it was shown that 
$$
	\xi' \Upsilon \circ \varpi = \xi' \colon P \otimes_{\zp} \qp \to P \otimes_{\zp} \qp.
$$
Our generation result allows us to remove the tensor product with $\qp$, which is to say to not have to work
modulo torsion in $P$.

\begin{theorem} \label{noqp}
	One has $\xi' \Upsilon \circ \varpi = \xi'$ as endomorphisms of $P$.
\end{theorem}

The key point is that there is an intermediate map $z^{\sharp}$, used in defining $\varpi$, and through which both sides of the above identity factor, sending compatible sequences of Manin symbols up the modular tower $X_1(Mp^r)$ to compatible sequences of zeta elements.  As constructed in \cite[Section 3.3]{fk-proof}, the map $z^{\sharp}$ is only defined rationally, which is to say that it takes values in $H^1(\Z[\frac{1}{N}],\T(1)) \otimes_{\La} Q(\La)$, where $\T$ is an inverse limit of ordinary parts of \'etale cohomology groups of the modular curves $X_1(Mp^r)$, and $Q(\La)$ is the quotient field of $\La$.  In Theorem \ref{zsharp}, we show that $z^{\sharp}$ can be defined to take values in $H^1(\Z[\frac{1}{N}],\T(1))$ on the eigenspaces relevant for our application.  
(Here, there is a restriction on the character that can be removed if one takes Eisenstein parts.)
For this purpose, we need that $H^1(\Z[\frac{1}{N}],\T(1))$ injects into its tensor product with $Q(\La)$,
so we show along the way that this cohomology group is $\La$-torsion free.

In the construction of $z^{\sharp}$, the zeta elements arise from norm compatible sequences of Beilinson elments in $H^2(Y_1(Mp^r)_{/\Z[\frac{1}{N}]},\zp(2))$, which is to say cup products of Siegel units $g_u$ on $Y_1(Mp^r)_{/\Z[\frac{1}{N}]}$.  In essence, the map takes a compatible sequence of Manin symbols $[u:v]$ to a sequence of cup products $g_u \cup g_v$.  That this map is well-defined is seen via a regulator computation.  One might hope for a direct construction of such a map, as Weierstrass units provide solutions to the unit equation that can be used to produce relations on Beilinson elements (see for instance \cite{brunault}).  One difficulty with doing this, which is also what leads to $z^{\sharp}$ not being defined integrally, is that the Siegel units only become true units after being raised to a power, or modified using an auxiliary integer $c$ as above.
This appears to necessitate working rationally, or at least restricting to a map which takes sequences of symbols ${}_{c,d}[u:v]$ to sequences of Beilinson elements ${}_c g_u \cup {}_d g_v$.  This causes further problems with performing a direct construction, as it is less clear what relations the latter Beilinson elements should satisfy.  On the other hand, the fact that $(c,d)$-symbols generate homology can be used to show that the map $z^{\sharp}$ is indeed defined integrally on homology itself, the regulator computation being immune to such issues.   

The derivation of Theorem \ref{noqp} from the integrality of $z^{\sharp}$ will be explained in the paper \cite{sh-extn}, which attempts to refine of the method of \cite{fk-proof}.  The survey paper \cite{fks} also contains an exposition of the conjecture of \cite{sh-Lfn} and the method of \cite{fk-proof}.

\begin{ack}
	The work of the first two authors was supported in part by the National Science Foundation under
	Grant No.~1001729.
	The work of the third author was supported in part by the National Science Foundation under Grant Nos.\ 
	1401122 and 1360583, and by a grant from the Simons Foundation (304824, R.S.).  
	He thanks Glenn Stevens, to whom this article is dedicated, for an inspiring discussion 
	on the topic of the integrality of the elements of Kato and Beilinson.
\end{ack}

\section{Generation by symbols} \label{generate}

We provide abstract definitions of spaces of modular symbols in Definitions \ref{modsyms} and \ref{cuspsyms}.  These are meant to model relative homology groups of modular curves, along with  their quotients.  That is, the relative homology groups have well-known presentations as modules over an algebra of diamond operators, with Manin symbols providing the generators.  We define a space of modular symbols to be a module generated by symbols with the same indexing set as the Manin symbols and satisfying at least the same relations.  We then formally define $(c,d)$-symbols in Definition \ref{cdsyms} using our Manin-type generators.  We consider the question of generation by $(c,d)$-symbols on individual primitive eigenspaces for action of the diamond-type operators.  
Our main result, Theorem \ref{cdgenerate}, describes exactly how close the $(c,d)$-symbols come to generating these eigenspaces.

To prove the theorem, we define eigensymbols as sums of Manin symbols that generate the eigenspace, the basic properties of which are studied in Section \ref{eigensymbols}.  One advantage of working with eigenspaces is that the $(c,d)$-versions of these eigensymbols are multiples of the corresponding unmodified eigensymbols, most of these multiples being by units, reducing the problem to a check that a relative few symbols are contained in the span of the $(c,d)$-symbols.  The proof of the result is contained in Section \ref{generation}.

For our applications in Section \ref{applications}, we are particularly interested in the generation by $(c,d)$-symbols of quotients of relative homology groups on which elements of Eisenstein ideals act trivially.  For this purpose, we artificially define Hecke-type operators in Section \ref{Hecke} that operate on the Manin-type symbols in the same manner that (adjoint) Hecke operators act on Manin symbols.  We restrict our definitions to what we actually require, which is to say that we only consider operators $U_{\ell}$ and $T_2$, as they suffice to obtain generation in the remaining primitive eigenspaces in Theorem \ref{U_theorem} and Proposition \ref{T_theorem}.

\subsection{The result} \label{result}

Let $p$ be an odd prime and $M$ a positive integer with $p \nmid M\varphi(M)$.
Set $N = M$ or $Mp$ (and suppose $N > 1$).  Let $\Delta = (\Z/N\Z)^{\times}/\langle-1\rangle$.  For
$a \in (\Z/N\Z)^{\times}$, let $\langle a \rangle$ denote the group element in $\zp[\Delta]$ corresponding to $a$.  

\begin{definition} \label{modsyms}
	A {\em space of level $N$ modular symbols} is a $\zp[\Delta]$-module $\tilde{H}$ spanned by symbols 
	$[u:v]$ for pairs of $u, v \in \Z/N\Z$ with $(u,v) = (1)$ that satisfy relations
	\begin{equation} \label{simplereln}
		[u:v] = [-u:-v] = -[-v:u], 
	\end{equation}
	as well as
	\begin{equation} \label{sumreln}
		[u:v] = [u:u+v] + [u+v:v],
	\end{equation}
	and, for every $a \in (\Z/N\Z)^{\times}$,
	\begin{equation} \label{diamondreln}
		\langle a \rangle[u:v] = [au:av].
	\end{equation}
\end{definition}	

\begin{definition} \label{cuspsyms}
	A {\em space of level $N$ cuspidal-at-zero modular symbols}  is a $\zp[\Delta]$-module $H$ spanned
	by symbols $[u:v]$ for pairs of nonzero $u, v \in \Z/N\Z$ with $(u,v) = (1)$ that satisfy relations
	\eqref{simplereln}, \eqref{sumreln} for $u \neq -v$, and \eqref{diamondreln}.
\end{definition}

Every space of level $N$ modular symbols has a subspace of a cuspidal-at-zero symbols.  Moreover, every
space of level $N$ cuspidal-at-zero symbols is also, rather artificially, a space of level $N$ modular symbols in which $[u:0] = 0$ for every nonzero $u \in (\Z/N\Z)^{\times}$, as the reader can easily check.  Let us fix a space $H$ of level $N$
modular symbols (which may or may not be cuspidal at zero in the sense just described).

\begin{definition} \label{cdsyms}
	For integers $c$ and $d$ greater than $1$ and prime to $6N$, the {\em $(c,d)$-symbol} 
	for the pair $(u,v)$ is defined as
	$$
		{}_{c,d}[u:v] = c^2d^2[u:v] - c^2[u:dv] - d^2[cu:v] + [cu:dv].
	$$
\end{definition}
	
Let $C$ denote the $\zp$-span of all $(c,d)$-symbols as we vary $c$ and $d$, which is a $\zp[\Delta]$-submodule of $H$.  Our goal is to compare $C$ with $H$ on primitive parts for the action of $(\Z/M\Z)^{\times}$.  To make this comparison, we work with individual $\Delta$-eigenspaces for $p$-adic characters.

So, fix a $p$-adic character $\theta \colon \Delta \to \C_p^{\times}$.  As before, let $\mc{O}_{\theta}$ denote the ring generated over $\zp$ by the values of $\theta$, and now let $\mc{O} = \zp[\mu_{\varphi(N)}]$, which contains $\mc{O}_{\theta}$.
Set $X_N = \Hom((\Z/N\Z)^{\times},\mc{O}^{\times})$.  Given an element of $X_N$ denoted by $\chi$, we will consistently use $\psi$ to denote $\theta\chi^{-1}$.

We have the following two types of eigenspaces.  First, for a $\zp[\Delta]$-module $D$, let 
$$
	D^{(\theta)} = D \otimes_{\zp[\Delta]} \mc{O}_{\theta},
$$ 
where $\zp[\Delta] \to \mc{O}_{\theta}$ is induced by $\theta$.
Second, let $e_{\theta}$ denote the idempotent
$$
	e_{\theta} = \frac{1}{\varphi(N)} \sum_{a \in (\Z/N\Z)^{\times}} \theta^{-1}(a) \langle a \rangle \in 
	\mc{O}_{\theta}[(\Z/N\Z)^{\times}].
$$
We then set 
$$
	D^{\theta} = e_{\theta}(D \otimes_{\zp} \mc{O}).
$$ 
Note that $D^{\theta} \cong D^{(\theta)} \otimes_{\mc{O}_{\theta}} \mc{O}$.  The advantage of working with eigenspaces $D^{\theta}$ is in the following definition of eigensymbols, the properties of which will be studied in the next subsection.

\begin{definition}
	For relatively prime divisors $g$ and $h$ of $N$ and $\chi \in X_N$, the {\em eigensymbol} 
	$\alpha_{\chi,\psi}^{g,h} \in H^{\theta}$ is defined as
	$$
		\alpha_{\chi,\psi}^{g,h}
		= \frac{1}{\varphi(N)^2} \sum_{a, b \in (\Z/N\Z)^{\times}} \chi^{-1}(a)\psi^{-1}(b)[ga:hb].
	$$
	The {\em $(c,d)$-eigensymbol} ${}_{c,d}\alpha_{\chi,\psi}^{g,h}$ is defined by the same formula after
	replacing $[ga:hb]$ by ${}_{c,d}[ga:hb]$.  We set $\alpha_{\chi,\psi} = \alpha_{\chi,\psi}^{1,1}$
	and ${}_{c,d} \alpha_{\chi,\psi} = {}_{c,d} \alpha_{\chi,\psi}^{1,1}$.
\end{definition}	

\begin{remark}
	By convention, if $H$ is cuspidal at zero, then the symbols $\alpha_{\chi,\psi}^{N,1}$ 
	with $\chi \in X_N$ are all zero.
\end{remark}

Let $\omega \colon (\Z/Mp\Z)^{\times} \to \zp^{\times}$ denote the Teichm\"uller character 
(factoring through $(\Z/p\Z)^{\times}$).  The following is the main result of this section.

\begin{theorem} \label{cdgenerate} \
	\begin{enumerate}
		\item[a.] If $p \ge 5$ and $\theta\omega^{-2}$ has conductor $N$, then $H^{\theta} = C^{\theta}$.
		\item[b.] If $N = Mp$ with $p \ge 5$, and $\theta\omega^{-2}$ has conductor $M > 1$, then 
		$H^{\theta} = C^{\theta} + \mc{O}\alpha_{\omega^2,\theta\omega^{-2}}^{1,p}$.
		\item[c.] If $N = Mp$ with $p = 3$, and $\theta$ has conductor $N$, then 
		$H^{\theta} = C^{\theta} + \mc{O}\alpha_{1,\theta} + \mc{O}\alpha_{1,\theta}^{N,1}$.
	\end{enumerate}
\end{theorem}

\subsection{Eigensymbols} \label{eigensymbols}

In this subsection, we study the eigensymbols $\alpha_{\chi,\psi}^{g,h}$, starting with the following
remarks.

\begin{remark}
	We have $\alpha_{\chi,\psi}^{g,h} = -\chi(-1)\alpha_{\psi,\chi}^{h,g}$.
\end{remark}

\begin{remark}
	We have
	$$
		\alpha_{\chi,\psi}^{g,h} = e_{\theta} \cdot 
		\frac{1}{\varphi(N)}\sum_{a \in (\Z/N\Z)^{\times}} \chi^{-1}(a)[ga:h].
	$$
\end{remark}

For $\chi \in X_N$, we let $f_{\chi}$ denote its conductor.
We will also view such a $\chi$ as a primitive Dirichlet character of modulus $f_{\chi}$.
Note that if $2$ divides $f_{\chi}$, then so does $4$.  For a divisor $D$ of $N$, set $Q_D = \frac{N}{D}$.

\begin{lemma} \label{eigrestrict}
	The element $\alpha_{\chi,\psi}^{g,h}$ is zero unless $f_{\chi}$ divides $Q_g$ and $f_{\psi}$ divides
	$Q_h$.
\end{lemma}

\begin{proof}	
	For $u, v \in (\Z/N\Z)^{\times}$, we have
    	$$
    		\frac{1}{\varphi(N)^2} \sum_{a, b \in (\Z/N\Z)^{\times}} \chi^{-1}(a)\psi^{-1}(b)[gau:hbv]
    		= \chi(u)\psi(v)\alpha_{\chi,\psi}^{g,h}.
    	$$
    	The right-hand side depends exactly (i.e., modulo no smaller integers) on the image of 
	$u$ in $(\Z/f_{\chi}\Z)^{\times}$ and the image of $v$ in $(\Z/f_{\psi}\Z)^{\times}$.
	Since the left-hand side depends only on $u$ modulo $Q_g$ and $v$ modulo $Q_h$, 
	the eigensymbol $\alpha_{\chi,\psi}^{g,h}$ can only be nonzero if $f_{\chi} \mid Q_g$ and $f_{\psi} \mid Q_h$.
\end{proof}

\begin{lemma} \label{multconst}
	If $u,v \in (\Z/N\Z)^{\times}$ and $g$ and $h$ are relatively prime divisors of $N$, then we have
	$$
		e_{\theta}[gu:hv] 
		= \sum_{\chi \in X_N} \chi(u)\psi(v)\alpha_{\chi,\psi}^{g,h}.
	$$
\end{lemma}

\begin{proof} 
	We have the following equalities:
	\begin{align*}
		\sum_{\chi \in X_N}\chi(u)\psi(v)\alpha_{\chi,\psi}^{g,h} &=
		\frac{1}{\varphi(N)}e_{\theta}  \sum_{\chi \in X_N}\chi(uv^{-1})\theta(v)
		\sum_{c \in (\Z/N\Z)^{\times}} \chi(c^{-1})[cg:h]
		\\&=
		\frac{1}{\varphi(N)}e_{\theta}  \sum_{\chi \in X_N} \sum_{c \in (\Z/N\Z)^{\times}} 
		\theta(v) \chi(c^{-1})[cguv^{-1}:h]\\
		&=  \frac{1}{\varphi(N)} e_{\theta}
		\theta(v)\langle v \rangle^{-1} 
		\sum_{c \in (\Z/N\Z)^{\times}} \sum_{\chi \in X_N} \chi(c^{-1})[cgu:hv]\\
		&= \frac{1}{\varphi(N)} e_{\theta} \varphi(N) [gu:hv] = e_{\theta}[gu:hv],
	\end{align*}
	where in the second to last step we have used that $e_{\theta}\langle v \rangle = \theta(v)$
	and that $\sum_{\chi \in X_N} \chi(c^{-1}) = 0$ unless $c = 1$.
\end{proof}

\begin{corollary} \label{eigspan}
	The elements $\alpha_{\chi,\psi}^{g,h}$ for $\chi \in X_N$ and relatively prime divisors 
	$g$ and $h$ of $N$ generate $H^{\theta}$ as an $\mc{O}$-module.  
\end{corollary}

Similarly, $C^{\theta}$ is spanned by the ${}_{c,d}\alpha_{\chi,\psi}^{g,h}$ as $c$, $d$, 
$\chi$, $g$, and $h$ vary over the relevant sets.   Moreover, note that
\begin{equation} \label{relatesym}
	{}_{c,d}\alpha_{\chi,\psi}^{g,h} = (c^2-\chi(c))(d^2-\psi(d))\alpha_{\chi,\psi}^{g,h}.
\end{equation}

\begin{remark} \label{eigokay}
	For $\chi \in X_N$ with $\chi \neq \omega^2$ and $\psi \neq \omega^2$, 
	we can choose $c, d$ prime to $6N$ such that $\chi(c) \not\equiv c^2 \bmod p$ and 
	$\psi(d) \not\equiv d^2 \bmod p$.  By \eqref{relatesym}, it follows that $\alpha_{\chi,\psi}^{g,h}$ is a unit multiple 
	of the resulting ${}_{c,d}\alpha_{\chi,\psi}^{g,h}$.  
\end{remark}

\subsection{Proof of generation} \label{generation}

In this subsection, we prove Theorem \ref{cdgenerate}.
Let us outline our strategy.
First, placing ourselves in the setting of the theorem for an odd prime $p$, 
we suppose in this subsection that $f = f_{\theta\omega^{-2}}$ is divisible by $M$ (so is $M$ or $Mp$), 
which forces $M$ to be odd or divisible by $4$.  Moreover, if $p = 3$, we suppose that $f = Mp$.
Either $2 \nmid f$ or $4 \mid f$, and in that $\theta$ is even, $f \ne 3,4$.
For brevity, for relatively prime divisors $g$ and $h$ of $N$, let us set 
$$
	\beta_{g,h} = \alpha_{\omega^2,\theta\omega^{-2}}^{g,h}
$$
and $\beta = \beta_{1,1} = \alpha_{\omega^2,\theta\omega^{-2}}$.  

By Remark \ref{eigokay}, we have that $\alpha_{\chi,\psi}^{g,h} \in C^{\theta}$
if $\chi \in X_N - \{ \omega^2, \theta\omega^{-2} \}$.
Let us use $A \subseteq C^{\theta}$ to denote the $\mc{O}$-span of all symbols
$\alpha_{\chi,\psi}^{g,h}$ with $\chi \notin \{ \omega^2,\theta\omega^{-2} \}$ and $g$ and $h$ relatively prime divisors  of $N$.  
By Corollary \ref{eigspan}, in order to prove Theorem \ref{cdgenerate}, 
we need only see that each $\alpha_{\chi,\psi}^{g,h}$ with $\chi \in \{ \omega^2,\theta\omega^{-2} \}$ lies in $A$ in case (i),
in $A + \mc{O}\beta_{1,p}$ in case (ii), and in 
$A + \mc{O}\beta + \mc{O}\beta_{N,1}$ in case (iii).

If $N = M$ and $p \ge 5$, we have $\chi \notin \{ \omega^2, \theta\omega^{-2} \}$ automatically since $\omega^2 \notin X_N$, so there is nothing more to show.  (It is not much harder to show that, for $p \ge 5$, Theorem \ref{cdgenerate}a holds for $N = M$ without the assumption $p \nmid \varphi(M)$ and without restriction on the conductor of $\theta$.)
So, suppose from now on that $N = Mp$.  The proof of Theorem \ref{cdgenerate} is contained in the lemmas that
follow.

\begin{lemma} \label{zero_terms}
	Let $g$ and $h$ be relatively prime divisors of $N$.
	If $p \nmid g$, then we have $\beta_{g,h} = 0$ unless $h = 1$, or $f = M$ and $h = p$. 
	If $p \ge 5$ and $p \mid g$, then $\beta_{g,h} = 0$.
\end{lemma}	 

\begin{proof}
	This is a a direct corollary of Lemma \ref{eigrestrict}.  For the second statement, we note for $p \ge 5$
	 that $f_{\omega^2} = p$ does not divide $Q_g$.
\end{proof}
	 
We deal below with the remaining terms.

\begin{lemma} \label{stupidlemma}
	Suppose that $g$ and $h$ are relatively prime divisors of $N$.	
	Let $\delta = 1$ if $N$ is odd or $gh$ is even and $\delta = 2$ otherwise.
	\begin{enumerate}
		\item[a.] There exist $a, b \in (\Z/N\Z)^{\times}$ such that $ag + bh = \delta$ in $\Z/N\Z$.
		\item[b.] Let $\rho \in X_N$ be nontrivial.  If $g \neq N$ and $h = \delta = 1$, then
		the numbers $a$ and $b$ of the first part can be chosen so that $\rho(b) \neq 1$ unless
		$f_{\rho} = 6\gcd(f_{\rho},g)$ and $3 \nmid g$.
	\end{enumerate}
\end{lemma}

\begin{proof}
	Part a is a simple application of the Chinese remainder theorem.  We focus on part b.  For its
	purposes, we can replace $N$ by $f_{\rho}$ and assume that $\rho$ is primitive of conductor $N$.
	
	Let $\ell$ be a prime dividing $Q_g$.
	Let $\epsilon = 0$ if $\ell$ is odd and $\epsilon = 1$ if $\ell = 2$,
	let $r \ge 1$ be such that $\ell^r$ exactly divides $N$, and let $s \ge 0$ be such that 
	$\ell^{s+\epsilon}$ exactly divides $g$ (noting for $\ell = 2$ that $g$ is even as $\delta = h = 1$).
	Since $\ell$ divides $Q_g$, we have $r > s+\epsilon$.
	It suffices to see that either $\rho_{\ell}  = \rho|_{(\Z/\ell^r\Z)^{\times}}$ takes on at least two different
	values on allowed values of $b$ modulo $\ell^r$, or at least one nontrivial value in the case that $Q_g$
	is a prime power.
	
	If $s \ge 1$, then primitivity of $\rho$ implies that $\rho_{\ell}(b)$ can be any one of 
	$\ell^{r-s-\epsilon-1}(\ell-1)$ different primitive 
	$\ell^{r-s-\epsilon}$th roots of unity (noting that $b-1$ is exactly divisible by $\ell^{s+\epsilon}$).  	
	This yields at least two values so long as $\ell$ is odd or $r > s+2$.  If $s = 0$, then $b$ can be taken to be 
	any prime-to-$\ell$ value modulo $\ell^r$ with $b \not\equiv 1 \bmod \ell^{1+\epsilon}$.  
	Again by primitivity, $\rho_{\ell}$ will take on at least two values on
	such $b$ if $\ell \ge 5$, $\ell = 3$ and $r \ge 2$, or $\ell = 2$ and $r \ge 3$.  
	If $\ell = 3$ and $r = 1$, then $b \equiv 2 \bmod 3$ and $\rho_3(b) = -1$.  For any $s \ge 0$, if $\ell = 2$ and
	$r = s+2$, we have $b \equiv 1+2^{s+1} \bmod 2^{s+2}$ and $\rho_2(b) = -1$.  
	This covers all cases but the excluded case $Q_g = 6$
	and $3 \nmid g$ (for which $\rho(b) = \rho_2(b)\rho_3(b) = 1$).
\end{proof}

Given relatively prime divisors $g$ and $h$ of $N$, we may
choose $a, b \in (\Z/N\Z)^{\times}$ with $ag + bh = \delta$
by first part of Lemma \ref{stupidlemma}.  
By applying Lemma \ref{multconst}, the formula 
$$
	[ga:hb] = [ga:\delta]+[\delta:hb]
$$ 
provides the identity
\begin{equation} \label{gensumeq}
	\sum_{\chi \in X_N} \chi(a)\psi(b) \alpha_{\chi,\psi}^{g,h} = \sum_{\chi \in X_N} \chi(a) \alpha_{\chi,\psi}^{g,\delta}
	+ \sum_{\chi \in X_N} \psi(b)\alpha_{\chi,\psi}^{\delta,h}.
\end{equation}
There are at most six terms in \eqref{gensumeq} not in $A$, two from each sum.
Using the antisymmetry of the eigensymbols for even characters, equation \eqref{gensumeq} yields
\begin{multline} \label{gensum6}
	\omega^2(a)\theta\omega^{-2}(b)\beta_{g,h} 
	- \omega^2(b)\theta\omega^{-2}(a)\beta_{h,g}
	\\ \equiv \omega^2(a)\beta_{g,\delta} 
	- \theta\omega^{-2}(a)\beta_{\delta,g}
	+ \theta\omega^{-2}(b)\beta_{\delta,h}
	- \omega^2(b)\beta_{h,\delta} \bmod A.
\end{multline}

\begin{lemma} \label{first_reduct}
	Suppose that $p \ge 5$.  For every proper divisor $g$ of $N$, we have $\beta_{g,1} \in A + \mc{O}\beta$.
\end{lemma}

\begin{proof}	
	Set $B = A + \mc{O}\beta$.  We may assume that $g > 1$, and by
	Lemma \ref{zero_terms}, that $p \nmid g$.
	
	Suppose first that $N$ is odd or $g$ is even, in which case we may choose $a, b \in (\Z/N\Z)^{\times}$ 
	such that $ag + b = 1$ by Lemma \ref{stupidlemma}a.  By Lemma \ref{zero_terms},
	we have $\beta_{1,g} = 0$.
	The last two terms of \eqref{gensum6} are in $B$ by definition.  So, 
	dividing out by $\omega^2(a)$,
	equation \eqref{gensum6} reduces to
	$$
		(\theta\omega^{-2}(b)-1)\beta_{g,1} \in B.
	$$
	If either $f \neq 6\gcd(f,g)$ or 
	$3 \mid g$, we may apply
	Lemma \ref{stupidlemma}b to
	choose $a$ such that $b = 1 - ag$ satisfies $\theta\omega^{-2}(b) \neq 1$ to conclude that
	$\beta_{g,1} \in B$.  In particular, we are done if $N$ is odd.
	
	Suppose, on the other hand, that $g$ is even, $f = 6\gcd(f,g)$, and $3 \nmid g$.
	We again apply Lemma \ref{stupidlemma}a to choose $a, b \in (\Z/N\Z)^{\times}$ 
	such that $ag + 3b = 1$.  Again by Lemma \ref{zero_terms}, we have 
	$\beta_{g,3} = \beta_{1,3} = 0$, as the facts that $3 \mid f$ and $M \mid f$ imply that $f \nmid \frac{N}{3}$. 
	We also have $\beta_{3,g} = \beta_{1,g} = 0$.  From \eqref{gensum6} with $h = 3$, we therefore obtain
	$$
		\omega^2(a)\beta_{g,1} - \omega^2(b)
		\beta_{3,1} \in A.
	$$
	It thus suffices to show that $\beta_{3,1} \in B$ in this case, so to prove the result for $g = 3$ and $N$ even.
	
	We are reduced to the case that $N$ is even and $g$ is odd.  Choose $a, b \in (\Z/N\Z)^{\times}$ 
	such that $ag + b = 2$, and consider \eqref{gensum6} with $\delta = 2$.
	Lemma \ref{eigrestrict} tells us that 
	$\beta_{1,2} = \beta_{g,2} = \beta_{1,g} = \beta_{2,g} = 0$, so this yields
	$$
		\omega^2(a)\beta_{g,1} - \beta_{2,1} \in A.
	$$
	Since $p \ge 5$, we have that $\omega^2$ 
	takes on at least two values on $a \not\equiv 2g^{-1} \bmod p$, 
	so both $\beta_{g,1}$ and $\beta_{2,1}$ are in $A$.  
\end{proof}	
	
\begin{lemma} \label{elt_relns}\
	\begin{enumerate}
		\item[a.]  Suppose that $p \ge 5$ and $f = M$.  For every divisor $g$ of $M$, we have 
		$\beta_{g,p} -\omega^2(g)\beta_{1,p} \in A + \mc{O}\beta$.
		\item[b.]  Suppose that $p = 3$ and $f = N$.  
		For every proper divisor $g$ of $N$, we have $\beta_{g,1} - \beta \in A$.  
	\end{enumerate}
\end{lemma}

\begin{proof}
	Let $h = p$ in part a and $h = 1$ in part b.  We may clearly suppose that $g > 1$.
	Choose $a, b \in (\Z/N\Z)^{\times}$ with $ag + bh = \delta$, with
	$\delta$ as in Lemma \ref{stupidlemma}a.
	In part a, we have $\beta_{\delta,g} = \beta_{p,g} = 
	\beta_{p,\delta} = 0$ by Lemma \ref{zero_terms},
	so \eqref{gensum6} yields
	$$
		\omega^2(a)\theta\omega^{-2}(b)\beta_{g,p} \equiv \omega^2(a)\beta_{g,\delta} + 
		\theta\omega^{-2}(b)\beta_{\delta,p} 
		\bmod A.
	$$
	If $\delta = 2$ (i.e., $g$ is odd and $N$ is even), 
	we have $\beta_{g,\delta} = 0$ and $a \equiv 2g^{-1} \bmod p$, so this reduces to
	$$
		\omega^2(2)\beta_{g,p} - 
		\omega^2(g)\beta_{2,p} \in A,
	$$
	which puts us back in the setting that $\delta = 1$.  
	If $\delta = 1$, then $a \equiv g^{-1} \bmod p$, and since $\beta_{g,1} \in A + \mc{O}\beta$ by 
	Lemma \ref{first_reduct}, we obtain that
	$\beta_{g,p} - \omega^2(g)\beta_{1,p} \in A+\mc{O}\beta$.

	For part b, we have $p = 3$ and $\omega^2 = 1$.  
	The module $H$ is trivial if $M = 1$, so we may take $M \ge 4$. 
	By \eqref{gensum6}, we have
	$$
		\theta(b)(\beta_{g,1}-\beta_{\delta,1}) \equiv \theta(a)(\beta_{1,g}-\beta_{\delta,g})
		+ (\beta_{g,\delta} - \beta_{1,\delta}) \bmod A.
	$$
	If $\delta = 2$ (so $g$ is odd and $N$ is even), then the last two terms are zero.  
	Since $g > 1$ and $f = N$ by assumption, the other two terms on the right
	are also zero, and we obtain $\beta_{g,1}-\beta_{2,1} \in A$.  
	If $\delta = 1$ (so $g$ is even or $N$ is odd), we obtain $(\theta(b)-1)(\beta_{g,1}-\beta) \in A$.
	We may choose $b$ to be any unit which is $1$ modulo $g$, and in particular so that 
	$\theta(b) \neq 1$ since $g$ properly divides $f = N$.
\end{proof}

Since we assumed that $f = N = Mp$ if $p = 3$, we are done in that case, and for $p \ge 5$, we are reduced to showing that $\beta \in A$ if $\theta \neq \omega^2$.  We consider $M > 1$ and $M = 1$ separately.

\begin{lemma} \label{M_not_1_case}
	If $p \ge 5$ and $M > 1$, then $\beta \in A$.
\end{lemma}
	
\begin{proof}
	Set $\epsilon = 1$ if $N$ is odd and $\epsilon = 2$ if $N$ is even.
	Using Lemma \ref{stupidlemma}a with $g = \epsilon p$ and $h = 1$, 
	choose $a, b \in (\Z/N\Z)^{\times}$ with $\epsilon pa+b = 1$.
	As $p \ge 5$, we have 
	$\beta_{\epsilon p,1} = 0$ by Lemma \ref{zero_terms}, and \eqref{gensum6} yields
	$$
		\theta\omega^{-2}(a)(\omega^2(b)-1)\beta_{1,\epsilon p}
		\equiv (\omega^2(b)-\theta\omega^{-2}(b))\beta \bmod A,
	$$
	but note that $b \equiv 1 \bmod p$, so this simplifies to
	$(1-\theta(b))\beta \in A$.
	By Lemma \ref{stupidlemma}b, 
	we may choose $b$ with $\theta(b) \neq 1$ so long as $M \neq 12$.

	If $M = 12$, instead choose $a, b \in (\Z/N\Z)^{\times}$ with $2a+b = 1$.  We then have 
	$b \equiv -1 \bmod 12$, and $12 \mid f$ forces $\theta|_{(\Z/12\Z)^{\times}}(-1) = 1$.  
	Let $k$ be the positive even integer with $k \le p-1$
	such that $\theta|_{(\Z/p\Z)^{\times}} = \omega^k$.
	Equation \eqref{gensum6} with $g = 2$ yields
	$$
		\omega^2(a)(\omega^{k-2}(b)-1)\beta_{2,1} \equiv (\omega^{k-2}(b)-\omega^2(b))\beta \bmod A.
	$$
	We then have $\beta \in A$ if there does not exist $w \in \F_p^{\times}$
	such that 
	$$
		G(a) = a^2((1-2a)^{k-2}-1) - w((1-2a)^{k-2}-(1-2a)^2)
	$$ 
	vanishes on all $a \in \F_p-\{0,\frac{1}{2}\}$.  As $G$ has degree $k \le p-1$,
	we see that $G$ cannot vanish on the the latter set unless $k = p-1$.  For $k = p-1$, note that
	$$
		G(a) = (1-2a)^{-2}(a^2(1-(1-2a)^2)-w(1-(1-2a)^4))
	$$
	on $\F_p-\{0,\frac{1}{2}\}$.  No polynomial of degree $4$ in $a$ can vanish on this set for $p > 5$.
	For $p = 5$ and $k = 4$, note that $G$ is independent of $w$ and $G(-1) = 3 \neq 0$.
  \end{proof}

\begin{lemma} \label{M_1_case}
	Suppose that $M = 1$ and $\theta \neq \omega^2$.
	The element $\beta$ is contained in the $\mc{O}$-span of the 
	symbols $\alpha_{\chi,\psi}$ with $\chi \in X_N - \{\omega^2,\theta\omega^{-2}\}$.
	In particular, $\beta \in A$.
\end{lemma}

\begin{proof}
	Suppose that $a$, $b$, and $a+b$ are all in $(\Z/p\Z)^{\times}$.
	From \eqref{gensumeq} with $g = h = 1$, we have
	\begin{equation} \label{sumeqn}
		\sum_{\chi \in X_N} (\chi(a)\psi(b)-\chi(a)\psi(a+b)-\chi(a+b)\psi(b))\alpha_{\chi,\psi} = 0.
	\end{equation}
	If $\chi \neq \psi$, the difference of the coefficients of $\alpha_{\chi,\psi}$ and $\alpha_{\psi,\chi}$ on the left of 
	\eqref{sumeqn} is
	$$
		(\chi(a)-\chi(a+b))(\psi(b)-\psi(a+b))
		- (\chi(b)-\chi(a+b))(\psi(a)-\psi(a+b)).
	$$
	It suffices to show that this is nonzero in the case $\chi = \omega^2$ (as $\omega^2$ is even,
	noting that $\alpha_{\omega^2,\omega^2} = 0$)
	for some choice of $a$ and $b$, as 
	we can then write $\beta$ as a sum of the $\alpha_{\chi,\psi}$
	with $\chi \notin \{ \omega^2, \theta\omega^{-2} \}$.  It is sufficient to consider the case
	$b = 1$, as is seen by dividing \eqref{sumeqn} by $\chi(b)\psi(b)$, so we suppose this from now on.
	
	Now write $\theta = \omega^k$ for an even $4 \le k \le p-1$, noting that $k \neq 2$ by assumption.
	Our difference is then
	$$
		F(a) = (a^2-(a+1)^2)(1-(a+1)^{k-2}) - (1-(a+1)^2)(a^{k-2}-(a+1)^{k-2})
	$$
	modulo $p$.
	The polynomial $F(a)$ has degree at most $k$ and is nonzero as $k \neq 2$ (e.g., note that $F'(0) = k$),
	so it cannot vanish identically on $\F_p - \{0,-1\}$ unless $k = p-1$.  
	If $k = p-1$, we can reduce our polynomial on $\F_p-\{0,-1\}$ to
	\begin{align*}
		(a^2-(a+1)^2)&(1-(a+1)^{p-3}) - (1-(a+1)^2)(a^{p-3}-(a+1)^{p-3}) 	\\
		&= a^{-2}(a+1)^{-2}(a^2(a^2-(a+1)^2)((a+1)^2-1) - (1-(a+1)^2)((a+1)^2-a^2)) \\
		&= -a^{-1}(a+1)^{-1}(a+2)(2a+1)(a-1)
	\end{align*}
	which is nonzero at $a = 2$ for $p \ge 7$.  Moreover, if $p = 5$, then $\beta = \alpha_{\omega^2,\omega^2}
	= 0$.
\end{proof}

\subsection{Hecke-type operators} \label{Hecke}

In this subsection, we take a minimalistic approach to ``Hecke operators'' on $H$, introducing them only as needed to prove generation, with a view towards ordinary and Eisenstein components of homology, on which our so-called Hecke operators will be actual Hecke operators.  For the purpose of our application, it is necessary to distinguish spaces that are cuspidal at zero, so we let $H$ either be a space of level $N$ modular symbols
or a space of cuspidal-at-zero modular symbols.  We suppose (without loss of generality) in this subsection that $N = Mp$.

The relationship between the following definition, and that of a $T_2$-operator below, and actual Hecke operators will be explained by the proof of Theorem \ref{eisthm}.

\begin{definition} \label{Uell}
	For a prime $\ell$ dividing $N$, a {\em $U_{\ell}$-operator} on $H$
	is a $\zp[\Delta]$-linear endomorphism
	$U_{\ell} \colon H \to H$ satisfying
	\begin{equation} \label{Uformula}
		U_{\ell}[\ell u:v] = \sum_{k=0}^{\ell-1} [u + k \tfrac{N}{\ell}:v].
	\end{equation}
	for all $u, v \in \Z/N\Z$ with $(\ell u,v) = (1)$, and $\ell u \neq 0$ if $H$ is cuspidal at zero. 
	We say that $H$ {\em has a $U_{\ell}$-operator} if there exists a $U_{\ell}$-operator on $H$.  
\end{definition}

Let $\ell$ denote a prime dividing $N$, and let $s \ge 1$ denote the $\ell$-adic valuation of $N$.
As before, $g$ and $h$ denote relatively prime divisors of $N$.

\begin{proposition} \label{U-operator}
	Suppose that $H$ has a $U_{\ell}$-operator and that $g$ and $h$ are not divisible by $\ell$.
	Let $t \le s$ be a positive integer such that $f_{\chi}$ divides $\frac{N}{\ell^t}$.
	If $t < s$, then
	$$
		U_{\ell}\alpha_{\chi,\psi}^{\ell^tg,h} = \ell \alpha_{\chi,\psi}^{\ell^{t-1}g,h},
	$$
	and if $t=s$ (and $\ell^s g \neq N$ if $H$ is cuspidal at zero), then
	$$
		(U_{\ell}-\chi^{-1}(\ell))\alpha_{\chi,\psi}^{\ell^sg,h} = (\ell-1) \alpha_{\chi,\psi}^{\ell^{s-1}g,h}.
	$$
\end{proposition}

\begin{proof}
	By Lemma \ref{eigrestrict}, we may suppose that $f_{\chi}$ divides $Q_{\ell^t g} = \frac{N}{\ell^t g}$.  We have
	\begin{equation} \label{U_formula}
		U_{\ell}\alpha_{\chi,\psi}^{\ell^tg,h} =  \sum_{k = 0}^{\ell-1}
		e_{\theta} \cdot \frac{1}{\varphi(N)}\sum_{a \in (\Z/N\Z)^{\times}} \chi^{-1}(a)
		[\ell^{t-1} g a + k \tfrac{N}{\ell}: h].
	\end{equation}
	Note that
	$$
		\ell^{t-1}ga + k \frac{N}{\ell} = \ell^{t-1}g\left(a + k\frac{N}{\ell^t g}\right),
	$$
	and $a+kQ_{\ell^t g}$ is prime to $Q_g$ except in the case $t = s$ 
	for the unique value of $k$ such that $\ell$ divides $a+kQ_{\ell^s g}$.  
	If $t < s$, we see that the $a+kQ_{\ell^tg}$ for a fixed $k$ run over all units in $\Z/Q_{\ell^{t-1}g}\Z$ with
	multiplicity $\ell^{t-1}g$ each.  As $f_{\chi} \mid Q_{\ell^t g}$, we therefore have
	$$
		\sum_{a \in (\Z/N\Z)^{\times}} \chi^{-1}(a)[\ell^{t-1} g a + k \tfrac{N}{\ell}: h]
		= \sum_{a \in (\Z/N\Z)^{\times}} \chi^{-1}(a)[\ell^{t-1}ga:h],
	$$
	proving the first statement.
	
	For $t = s$, the $a+kQ_{\ell^sg}$ for all $k$ and $a$ such that $\ell \nmid (a+kQ_{\ell^s g})$
	run over $(\Z/Q_{\ell^{s-1}g}\Z)^{\times}$ with multiplicity $(\ell-1)\ell^{s-1}g$, similarly providing
	the $(\ell-1)\alpha_{\chi,\psi}^{\ell^{s-1}g,h}$ in the second statement.  
	On the other hand, the $\ell^{-1}(a+kQ_{\ell^tg})$ for pairs $(k,a)$ with $\ell \mid (a+kQ_{\ell^st})$
	run over $(\Z/Q_{\ell^sg}\Z)^{\times}$ with constant multiplicity $\ell^s g$.  If $a' \in (\Z/N\Z)^{\times}$
	is such that $a' \equiv \frac{a+kQ_{\ell^sg}}{\ell} \bmod Q_{\ell^s g}$, then $\chi(\ell a') = \chi(a)$,
	so the sum of the corresponding terms is the desired value $\chi(\ell)^{-1}\alpha_{\chi,\psi}^{\ell^sg,h}$.
\end{proof}

In the following nearly immediate corollaries, if $H$ is cuspidal at zero, we implicitly exclude those elements $\alpha_{\chi,\psi}^{g,h}$ that occur with either $g$ or $h$ equal to $N$.

\begin{corollary}
	Suppose that $H$ has a $U_{\ell}$-operator that acts on $H^{\theta}$ as multiplication by an element of
	$\mc{O}^{\times}$.  Then 
	$H^{\theta}$
	is the $\mc{O}$-span of the elements of the form $\alpha_{\chi,\psi}^{g,h}$ and
 	$\alpha_{\chi,\psi}^{\ell^sg,h}$ with $\ell \nmid gh$.
\end{corollary}

Noting that $\ell - 1 \in \zp^{\times}$, and that if $\chi(\ell) \neq 1$, then $1-\chi^{-1}(\ell)$ is a unit, we have the 
following.

\begin{corollary} \label{triv_ell_act}
	Suppose that $H$ has a trivial $U_{\ell}$-operator (i.e., $U_{\ell}-1$ acts as zero on $H$). 
	Then $H^{\theta}$ is the $\mc{O}$-span of the elements of the form $\alpha_{\chi,\psi}^{g,h}$ 
	with $\ell \nmid gh$ and $\alpha_{\chi,\psi}^{\ell^sg,h}$ with $\ell \nmid gh$ and $\chi(\ell) = 1$.
	Moreover, if $\ell \nmid f_{\chi}gh$ and $\chi(\ell) = 1$, then $\alpha_{\chi,\psi}^{g,h} = 0$
	(so long as $\ell^s g \neq N$ if $H$ is cuspidal at zero).
\end{corollary}

We may now prove the main theorem of this section.

\begin{theorem} \label{U_theorem}
	Suppose that $f = f_{\theta\omega^{-2}}$ is divisible by $M$.  
	If any of the following conditions hold, then $H^{\theta} = C^{\theta}$:
	\begin{enumerate}
		\item[i.] $p \ge 5$ and $f = N$,
		\item[ii.] $p \ge 5$, $f = M$, $\theta|_{(\Z/M\Z)^{\times}}(p) \neq 1$, and 
		there is a trivial $U_p$-operator on 
		$H$,
		\item[iii.] for some prime $\ell$ dividing $M$, there is a trivial $U_{\ell}$-operator on $H$
		(and $f = N$ and $H$ is cuspidal at zero if $p = 3$).
	\end{enumerate}
\end{theorem}

\begin{proof}
	By Theorem \ref{cdgenerate}, in the case $p \ge 5$ (resp., $p = 3$),
	we need only show that  $\beta_{1,p} \in A$ if $f = M$ (resp., $\beta \in A$).  
	In particular, we already have (i).
	If $\theta|_{(\Z/M\Z)^{\times}}(p) \neq 1$ and $U_p$ acts trivially, 
	we have $\beta_{1,p} \in A + \mc{O}\beta$ by Proposition \ref{U-operator}
	for $U_p$.  Since $\beta \in A$ by Lemmas \ref{M_not_1_case} and \ref{M_1_case},
	we have (ii).  
	
	If $p \ge 5$ and $f = M$ (resp., $p = 3$ and $f = N$), then for $h = p$ (resp., $h = 1$), 
	Lemma \ref{elt_relns} tells us that
	$$
		\omega(\ell)^{-2}\beta_{\ell,h}-\beta_{1,h} 
		\in A
	$$
	for any prime $\ell$ dividing $M$.  If $\ell^2$ divides $M$ (resp., $\ell$ exactly divides $M$ and
	$\ell \not\equiv -1 \bmod p$) and $U_{\ell}$ acts trivially, then Proposition \ref{U-operator} yields that
	$w\beta_{1,h} \in A$,
	where $w \equiv \frac{1-\ell}{\ell} \bmod p$ (resp., $w \equiv \frac{-\ell}{\ell+1} \bmod p$).
	Finally, if $\ell \equiv -1 \bmod p$ exactly divides $M$, then since $\omega^2(\ell) = 1$ and 
	$U_{\ell}$ acts trivially, Corollary \ref{triv_ell_act} tells us that
	$\beta_{1,h} = 0$.
\end{proof}

We can treat cuspidal-at-zero symbols in the case $M = 1$ and $\theta = \omega^2$ by using a $T_2$-operator.

\begin{definition}
	Suppose that $N$ is odd.  A {\em $T_2$-operator} on $H$ is a $\zp[\Delta]$-linear endomorphism of $H$
	such that
	\begin{equation} \label{T2dual}
		\langle 2 \rangle T_2[u:v] = [2u:v] + [2u:u+v]  + [u+v:2v] +  [u:2v] 
	\end{equation}
	for all $u, v \in \Z/N\Z$ with $(u,v) = (1)$ (and $u$, $v$, and $u+v$ nonzero if
	$H$ is cuspidal at zero).  
\end{definition}

\begin{proposition} \label{T_theorem}
	Suppose that $M = 1$, that $\theta = \omega^2$, and that $H$ is cuspidal at zero.
	If $T_2$ acts as $1+2\omega^{-2}(2)$ on $H^{\omega^2}$, then $H^{\omega^2} = C^{\omega^2}$.
\end{proposition}

\begin{proof}
	Since $H$ is cuspidal at zero, it clearly suffices show that $\beta \in A$.
	Let $a, b \in (\Z/N\Z)^{\times}$ be such that $a+b = 1$.
	Note that $e_{\omega^2}(\langle 2 \rangle T_2[a:b] - 2[a:b] - [2a:2b]) = 0$ by our assumption on $T_2$.
	Applying Lemma \ref{multconst}  
	to \eqref{T2dual} with $u = a$ and $v = b$, we therefore have that
	$$
		\sum_{\chi \in X_N} (2\chi(a)\psi(b) + \chi(2a)\psi(2b) -\chi(2a)\psi(b) - \chi(2a) 
		- \psi(2b) - \chi(a)\psi(2b))
		\alpha_{\chi,\psi} = 0,
	$$
	where as before $\psi$ is taken to be $\theta\chi^{-1}$.
	Since $\theta = \omega^2$, we need only see that the difference 
	$$
		(\omega^2(a)-\omega^2(2a)-1)
		- (\omega^2(b)-\omega^2(2b)-1)
		\equiv 3((1-a)^2-a^2)  \bmod p
	$$
	of the coefficients of $\alpha_{\omega^2,1} = \beta$
	and $\alpha_{1,\omega^2} = -\beta$ can be nonzero modulo $p$.  For $p \ge 5$, this 
	occurs for $a = -1$, and for $p = 3$, the group $H$ is trivial.
\end{proof}

\begin{remark}
	The equality in Proposition \ref{T_theorem} does not always hold without the cuspidal-at-zero condition,
	even on ``Eisenstein quotients'' involving conditions for all Hecke operators, rather than just $T_2$.
\end{remark}

\section{Application to modular curves} \label{applications}

In Section \ref{homology}, we first apply the abstract study of Section \ref{generate} to prove Theorems \ref{mainthm} and \ref{eisthm} on the generation of relative homology groups of modular curves and their Eisenstein parts by the modifications of Manin symbols that we have termed $(c,d)$-symbols.  An application of Nakayama's lemma provides analogous results, Propositions \ref{ordcoh} and \ref{spancoh}, for inverse limits of the ordinary parts and Eisenstein components of these relative homology groups up a tower of modular curves of increasing $p$-power dividing the level.  

As outlined in the introduction, we apply these generation results in Section \ref{zeta} to prove the integrality of two types of maps taking compatible systems of Manin symbols to compatible systems of zeta elements: see Theorems \ref{zinfty} and \ref{zsharp}.  These maps, first defined in \cite{fk-proof}, a priori only take values in a tensor product of a cohomology group with the quotient field of an Iwasawa algebra.  However, their values on compatible systems of $(c,d)$-symbols lie in the (integral) image of the actual cohomology group in the (rational) tensor product.  By the generation results of Section \ref{generate}, the theorems are then reduced to the torsion-freeness of the Iwasawa modules that are the cohomology groups in question, and this is the content of Lemmas \ref{Iwtorsfree} and \ref{torsfree}.

\subsection{Homology of modular curves} \label{homology}

We explain how the results of Section \ref{generate} apply to the homology groups of modular curves.  For this, fix an odd prime $p$, an integer $M$ with $p \nmid M\varphi(M)$, and $r \ge 1$.  Set $N_r = Mp^r$, and suppose additionally that $N_r \ge 4$.  We set $N = N_1$.  Let $G_r = (\Z/N_r\Z)^{\times}/\langle -1 \rangle$, let $\Delta = G_1$, and let $\Gamma_r = \ker(G_r \to \Delta)$.  Set $\mc{O} = \zp[\mu_{\varphi(M)}]$.

Let $C_1(N_r) = \Gamma_1(Np^r) \backslash \mb{P}^1(\Q)$ denote the set of cusps in $X_1(N_r)$.  The relative homology group
$$
	\tilde{H}_r = H_1(X_1(N_r),C_1(N_r),\Z)
$$ 
is generated by the classes $\{ x \to y \}_r$  of images of geodesics in the extended upper half-plane 
from a cusp $x \in \mb{P}^1(\Q)$ to a cusp $y \in \mb{P}^1(\Q)$.

\begin{definition}
	For $u, v \in \Z/N_r\Z$ with $(u,v) = (1)$, choose 
	$\left(\begin{smallmatrix} a&b \\ c&d \end{smallmatrix}\right) \in \SL_2(\Z)$
	with $(u,v) = (c,d) \bmod N_r\Z^2$.  The {\em Manin symbol} of level $N_r$ attached to the pair $(u,v)$ is
	$$
		[u:v]_r = \left\{ \frac{-d}{bN_r} \to \frac{-c}{aN_r} \right\}_r.
	$$
\end{definition}

\begin{remark} \label{modified}
	Our ``Manin symbols'' are actually Manin's after application of the Atkin-Lehner involution of level $N_r$.
	As a result, we have 
	$$ 
		\langle a \rangle_r^{-1}[u:v]_r = [au:av]_r
	$$
	for $a \in G_r$, where $\langle a \rangle_r$ denotes the diamond operator acting on $\tilde{H}_r$.
\end{remark}

It is well-known that $\tilde{H}_r$ is a space of modular symbols of level $N_r$ in the sense previously defined.  In fact, we have the following, which follows easily from \cite[Theorem 1.9]{manin}.

\begin{theorem}[Manin] \label{maninthm}
 	 The space $\tilde{H}_r$ has a presentation as a $\zp[G_r]$-module on the generators $[u:v]_r$ with 
	 relations as in \eqref{simplereln}, \eqref{sumreln}, and \eqref{diamondreln}, where $\langle a \rangle$ 
	 in \eqref{diamondreln} for $a \in (\Z/N_r\Z)^{\times}$ is taken to be the inverse diamond 
	 operator $\langle a \rangle_r^{-1}$ acting on $\tilde{H}_r$.
\end{theorem}

Let 
$$
	H_r = H_1(X_1(N_r),C_1^0(N_r),\zp),
$$
where $C_1^0(N_r) \subset C_1(N_r)$ is the subset of cusps not lying over the zero cusp of the modular
curve $X_0(N_r)$.  This is the space of cuspidal-at-zero symbols in $\tilde{H}_r$.
Let $\tilde{C}_r$ and $C_r$ denote the $\zp$-submodules of $\tilde{H}_r$ and $H_r$, respectively, generated by $(c,d)$-symbols ${}_{c,d}[u:v]_r$ as we vary $c$ and $d$ (with $u, v \neq 0$ for $H_r$).

Fix a $p$-adic character $\theta$ of $\Delta$ as before.  Theorem \ref{cdgenerate} applies directly to $H_1$ and $\tilde{H}_1$ for the symbols $[u:v]_1$, yielding Theorem \ref{mainthm} and its analogue for $H_1$.  The proof of Theorem \ref{eisthm} is only slightly more involved.

\begin{proof}[Proof of Theorem \ref{eisthm} and Remark \ref{mainremark}]
	The Hecke operators $U_{\ell}$ for $\ell \mid N$ (resp., $T_2$ for $N$ odd) on $\tilde{H}_1$ satisfy 
	equation \eqref{Uformula} 
	(resp., \eqref{T2dual}) by \cite[p.\ 264]{sh-Lfn} (resp., \cite[Prop.~20]{merel}, after taking into account the modification
	to said formula incurred by Remark \ref{modified}).  
	Let $I_1$ be the Eisenstein ideal of the introduction, generated by 
	by $U_{\ell} - 1$ for $\ell \mid N$ and $T_{\ell}-1-\ell\langle \ell \rangle$ for  
	$\ell \nmid N$.
	To prove Theorem \ref{eisthm}, we may take $\tilde{H}_1/I_1\tilde{H}_1$
	as $H$ in Theorem \ref{U_theorem}.  For Remark \ref{mainremark}, we may take 
	$H_1/I_1H_1$ as $H$ in Theorem \ref{U_theorem} and
	in Proposition \ref{T_theorem} if $\theta = \omega^2$.	
	Consequently, $H^{(\theta)}$ is generated by the images of the ${}_{c,d}[u:v]_1$.  
	Applying Nakayama's lemma, we obtain Theorem \ref{eisthm} for $\tilde{H}_1$, and its analogue for $H_1$.
\end{proof}

Let $G = \invlim_r G_r$ and $\Gamma = \invlim_r \Gamma_r$, and set $\tilde{\La} = \zp\ps{G}$ and 
$\La = \zp\ps{\Gamma}$.  Let us use $\ord$ to denote the ordinary part of a Hecke algebra or module, the maximal 
direct summand on which (a compatible system of Hecke operators) $U_p$ acts invertibly.   
We will identify elements with their projections to ordinary parts (via Hida's ordinary idempotent).
We then have $\La$-modules $\mc{H} = \invlim_r H_r^{\ord}$, $\tilde{\mc{H}} = \invlim_r \tilde{H}_r^{\ord}$, $\mc{C} = \invlim_r C_r^{\ord}$, and $\tilde{\mc{C}} = \invlim_r \tilde{C}_r^{\ord}$, with $a \in \Gamma$ acting as
the inverse diamond operator $\langle a \rangle^{-1}$.  As in \cite[Lemma 3.2]{sh-Lfn}, we have the compatibility
of Manin symbols in the following definition.

\begin{definition} \label{invlimcd}
	For integers $u$ and $v$ with $p \nmid v$ and $(u,v,M) = 1$, let 
	$$
		(u:v) = ([p^{r-1} u : v]_r)_r \in \tilde{\mc{H}},
	$$
	and given integers $c, d > 1$ prime to $6N$, set
	$$
		{}_{c,d}(u:v) = c^2d^2(u:v) - d^2(cu:v) - c^2(u:dv) + (cu:dv) \in \tilde{\mc{H}}.
	$$
\end{definition}

As a $\tilde{\La}$-module, $\tilde{\mc{H}}$ is generated by the elements $(u:v)$ for $u, v \in \Z$
as above \cite[Lemma 3.2.5]{fk-proof}.   The proof is simple: the image of $(u:v)$ in 
$\tilde{H}_1^{\ord}$ is the ordinary projection of $[u:v]_1$.  Since $\tilde{H}_1^{\ord}$ is isomorphic by Hida theory to the group of $\Gamma$-coinvariants of $\tilde{\mc{H}}$ (see 
\cite[Theorem 1.4.3]{ohta-ord}), we may apply Nakayama's lemma.

The $\tilde{\La}$-submodule $\mc{H}$ of $\tilde{\mc{H}}$ is generated by those elements $(u:v)$ with $N \nmid u$.
We let $\tilde{\mc{C}}$ denote the $\tilde{\La}$-submodule of $\tilde{\mc{H}}$ generated by all $(c,d)$-symbols 
${}_{c,d}(u:v)$, and we let $\mc{C}$ be the $\tilde{\La}$-submodule generated by those with $N \nmid u$.

We let $\La_{\theta} = \tilde{\La}^{(\theta)} \cong \mc{O}_{\theta}\ps{\Gamma}$.  
We again set $f = f_{\theta\omega^{-2}}$.  We take $p \ge 5$ in the following two results for simplicity of presentation.
They are immediate consequences of Theorems \ref{mainthm} and \ref{eisthm} (and their cuspidal-at-zero analogues) upon application of Nakayama's lemma.

\begin{proposition} \label{ordcoh}
	If $M \mid f$ and $\theta \neq \omega^2$, then
	\begin{eqnarray*}
		\mc{H}^{(\theta)} = \mc{C}^{(\theta)} + \La_{\theta}e_{\theta}(p:1)
		&\text{and}& \tilde{\mc{H}}^{(\theta)} = \tilde{\mc{C}}^{(\theta)} +  \La_{\theta}e_{\theta}(p:1).
	\end{eqnarray*}
	Moreover, if $f = N$, then $\mc{H}^{(\theta)} = \mc{C}^{(\theta)}$
	and $\tilde{\mc{H}}^{(\theta)} = \tilde{\mc{C}}^{(\theta)}$.
\end{proposition}

Let $\mf{H}$
denote Hida's ordinary Hecke algebra acting as $\tilde{\La}$-linear endomorphisms of $\tilde{\mc{H}}$,
and let $\mf{h}$ denote the cuspidal Hecke algebra acting on $\mc{H}$.
Note that $\mf{H}^{(\theta)}$ is a free $\La_{\theta}$-module, again by Hida theory.
Let $I$ denote the Eisenstein ideal of $\mf{H}$ (and its image in $\mf{h}$) that is generated by
$T_{\ell} - 1 - \ell\langle \ell \rangle$ for primes 
$\ell \nmid N$ and $U_{\ell} - 1$ for primes $\ell \mid N$. 
Let $\mf{m}$ denote the unique maximal ideal of $\mf{H}^{(\theta)}$ containing $I\mf{H}^{(\theta)}$.  We write
the localization of $\mf{H}^{(\theta)}$ at $\mf{m}$ more briefly as $\mf{H}_{\mf{m}}$, and similarly for
$\mf{H}$-modules.

\begin{proposition} \label{spancoh}
	If $M \mid f$, then $\mc{H}_{\mf{m}}$ is equal to $\mc{C}_{\mf{m}}$.  If $\theta \neq \omega^2$ as well,
	 then $\tilde{\mc{H}}_{\mf{m}}$ is equal to $\tilde{\mc{C}}_{\mf{m}}$.
\end{proposition}

\subsection{Zeta elements} \label{zeta}

In this subsection, we consider the integrality of maps that take Manin symbols to zeta elements.  These maps were defined rationally in \cite[Section 3]{fk-proof}.  We suppose that $p \ge 5$. 
 
For $s \ge 0$, let $\Z_s$ denote the ring of integers of $\Q_s$, the maximal $p$-extension of $\Q$ in $\Q(\mu_{p^s})$, and let $\Z_{\infty}$ denote the ring of integers of the cyclotomic $\zp$-extension $\Q_{\infty}$ of $\Q$.  
Set 
\begin{eqnarray*}
	\T_r = H^1(X_1(N_r)_{/\overline{\Q}},\zp(1))^{\ord} &\mr{and}& 
	\tT_r = H^1(Y_1(N_r)_{/\overline{\Q}},\zp(1))^{\ord},
\end{eqnarray*}
and let $\T$ and $\tT$ denote the corresponding inverse limits under trace maps.  
We let Hida's Hecke algebras $\mf{h}$ and
$\mf{H}$ act on $\T$ and $\tT$, respectively, via the adjoint actions of Hecke operators (and here, ordinary parts
have been taken with respect to the adjoint action of $U_p$).
For any compact $\zp\ps{G_{\Q}}$-module $\mc{A}$ that is unramified outside of $N$, set
$$
	H^1_{\Iw}(\Z_{\infty}[\tfrac{1}{N}],\mc{A}) = \invlim_s H^1(\Z_s[\tfrac{1}{N}],\mc{A}).
$$

\begin{definition} \label{zetadef}
	For $r, s \ge 0$ and $u, v \in \Z/N_r\Z$ with $(u,v) = (1)$ (and $u, v \neq 0$ if $s = 0$) and $c,d >1$ and
	prime to $6N$, 
	we define ${}_{c,d}z_{r,s}(u:v)$ to be the image
	under the norm and Hochschild-Serre maps
	$$
		H^2(Y(p^s,N_{r+s})_{/\Z[\frac{1}{N}]},\zp(2)) \to H^2(Y_1(N_r)_{/\Z_s[\frac{1}{N}]},\zp(2))
		\xrightarrow{\mr{HS}} H^1(\Z_s[\tfrac{1}{N}],H^1(Y_1(N_r)_{/\qbar},\zp(2)))
	$$
	of the cup product 
	${}_c g_{\frac{w}{p^s},\frac{y}{N_{r+s}}} \cup {}_d g_{\frac{x}{p^s},\frac{z}{N_{r+s}}}$ 
	of Siegel 
	units on $Y(p^s,N_{r+s})_{/\Z[\frac{1}{N}]}$, where $\left( \begin{smallmatrix} w&x\\y&z  
	\end{smallmatrix} \right) \in \mr{SL}_2(\Z)$ with $u = y \bmod N_r$ and $v = z \bmod N_r$.
\end{definition}

\begin{remark}
	In defining ${}_{c,d}z_{r,s}(u:v)$, 
	we have taken corestriction from $\Z[\mu_{p^s}]$ to $\Z_s$ as part of the first map.  The latter
	discards information but is sufficient for the purposes of studying
	the conjecture of \cite{sh-Lfn}. 
\end{remark}

\begin{remark}
	For $\alpha, \beta \in \Q$ not both zero and with denominators dividing $L > 1$, there is a rationally-defined 
	Siegel unit $g_{\alpha,\beta} \in \mc{O}(Y(L)_{/\Z[\frac{1}{L}]})^{\times} \otimes_{\Z} \Q$ which satisfies
	${}_c g_{\alpha,\beta} = g_{\alpha,\beta}^{c^2} g_{c\alpha,c\beta}^{-1} \in \mc{O}(Y(L)_{/\Z[\frac{1}{L}]})^{\times}$
	for $c > 1$ prime to $6L$.  
	Taking a cup product followed by a norm of two Siegel units attached to a pair $(u,v)$ as in Definition \ref{zetadef}, 
	we obtain an element 
	$$
		z_{r,s}(u:v) \in H^1(\Z_s[\tfrac{1}{N}],H^1(Y_1(N_r)_{/\qbar},\qp(2))).
	$$  
	For $s = 0$, the relationship between this zeta element and the $(c,d)$-version is given in the
	latter cohomology group by
	$$
		{}_{c,d} z_{r,0}(u:v) = c^2d^2z_{r,0}(u:v) - d^2z_{r,0}(cu:v) - 
		c^2z_{r,0}(u:dv) + z_{r,0}(cu:dv),
	$$
	mirroring our definition of $(c,d)$-symbols in Definition \ref{cdsyms}.  For $s > 0$, the relationship between the
	zeta elements $z_{r,s}(u:v)$ and ${}_{c,d}z_{r,s}(u:v)$ is likewise analogous to the relationship
	between the symbols $(u:v)$ and the symbols ${}_{c,d}(u:v]$ defined below in \eqref{fancycd}.
\end{remark}

\begin{remark}
	We have the following norm compatibilities among zeta elements.
	\begin{enum}
		\item[i.] For a fixed $s \ge 0$, the elements ${}_{c,d}z_{r,s}(p^{r-1}u:v)$ with $r \ge 1$ are
	compatible with the norm maps induced by quotients of modular curves.  (This is seen through the
	observation that
	${}_{c,d}z_{r+1,s}(p^r u:v)$ arises from a norm of a cup product of two Siegel units, the first of 
	which is already a unit on $Y(p^s,N_{r+s})_{/\Z[\frac{1}{N}]}$.)
		\item[ii.] For a fixed $r \ge 1$, the elements ${}_{c,d}z_{r,s}(u:v)$ are compatible for $s \ge 1$
	under corestriction maps for the ring extensions, while the corestriction from $\Z_s[\frac{1}{N}]$ to 
	$\Z[\frac{1}{N}]$ takes ${}_{c,d}z_{r,s}(u:v)$ to $(1-U_p){}_{c,d}z_{r,0}(u:v)$
	(see \cite[Prop.~2.4.4]{fk-proof}).
	\end{enum}
\end{remark}

Let us use ${}_{c,d}z_{r,s}(u:v)_{\theta}$ to denote the projection of ${}_{c,d}z_{r,s}(u:v)$ to 
$H^1(\Z_s[\frac{1}{N}],\tT_r^{(\theta)}(1))$.  We will use the same notation to denote its image in
the Eisenstein component $H^1(\Z_s[\frac{1}{N}],(\tT_r)_{\mf{m}}(1))$.

\begin{definition}
	For $c, d > 1$ with $(c,d,6N) = 1$ and $u, v \in \Z$ with $(u,v,M) = 1$ and $p \nmid v$, we set
	$$
		{}_{c,d} z(u:v)_{\theta} = ({}_{c,d}z_{r,s}(p^{r-1}u:v)_{\theta})_{r,s \ge 1} \in 
		H^1_{\Iw}(\Z_{\infty}[\tfrac{1}{N}],\tT^{(\theta)}(1)). 
	$$
	If $u \not\equiv 0 \bmod N$, we also set
	$$
		{}_{c,d} z^{\sharp}(u:v)_{\theta} =
		 ({}_{c,d}z_{r,0}(p^{r-1}u:v)_{\theta})_{r \ge 1} \in H^1(\Z[\tfrac{1}{N}],\tT^{(\theta)}(1)).
	$$ 
\end{definition}

\begin{remark}
	The corestriction of ${}_{c,d} z(u:v)_{\theta}$ to $H^1(\Z[\tfrac{1}{N}],\tT^{(\theta)}(1))$ is 
	$(1-U_p)z^{\sharp}(u:v)_{\theta}$
	for $u \not\equiv 0 \bmod N$.
\end{remark}

In what follows, $\La$ should typically be thought of as an Iwasawa algebra of elements of $\Gal(\Q_{\infty}/\Q)$, while $\La_{\theta}$ should be considered as an algebra of inverses of diamond operators.
Let $Q(O)$ denote the quotient field of an integral domain $O$.

\begin{remark}
	In \cite[Prop.~3.1.3]{fk-proof},  the $\La$-module
	$H^1_{\Iw}(\Z[\mu_{p^{\infty}},\tfrac{1}{N}],H^1(X_1(N_r)_{/\qbar},\zp)(2))$ is shown to be
	torsion-free.  The proof of the following lemma yields the analogue on ordinary parts
	in the inverse limit over $r$, which is to say that $H^1_{\Iw}(\Z[\mu_{p^{\infty}},\tfrac{1}{N}],\T^{(\theta)}(1))$
	is $\La \cotimes{\zp} \La_{\theta}$-torsion free, without assumption on $\theta$.  However,
	we work with the cohomology of $\Z_{\infty}[\frac{1}{N}]$, as opposed to $\Z[\mu_{p^{\infty}},\frac{1}{N}]$, 
	in order to obtain
	torsion-freeness with coefficients in the larger module $\tT^{(\theta)}(1)$.
\end{remark}

 In the proofs of the results that follow, though not the statements, we will omit superscripts ``${}^{(\theta)}$'' denoting
eigenspaces and subscripts ``${}_{\mf{m}}$'' denoting Eisenstein parts to lessen notation.

\begin{lemma}\ \label{Iwtorsfree}
	The $\La \cozp \La_{\theta}$-module $H^1_{\Iw}(\Z_{\infty}[\tfrac{1}{N}],\tT^{(\theta)}(1))$ is
	torsion-free.  
\end{lemma}

\begin{proof}
	Note first that $H^1_{\Iw}(\Z_{\infty}[\tfrac{1}{N}],\T(1))$ has no nonzero $\Lambda_{\theta}$-torsion 
	as zeroth Iwasawa cohomology groups are trivial.
	We modify the argument of \cite[Prop.~3.1.3(1)]{fk-proof} in order to see that
	$H^1_{\Iw}(\Z_{\infty}[\tfrac{1}{N}],\T(1))$ has no $\La \cozp \La_{\theta}$-torsion.  
	Let $\La^{\iota}$ denote $\La$ with the inverse action 
	of elements of the absolute Galois group $G_{\Q}$, and set $U = \La^{\iota} \cozp \T(1)$.
	Take a nonzero $\nu \in \La \cozp \La_{\theta}$.  Any $\nu$-torsion in the $\La_{\theta}$-torsion free 
	$$
		H^1_{\Iw}(\Z_{\infty}[\tfrac{1}{N}],\T(1)) \cong H^1(\Z[\tfrac{1}{N}],U)
	$$ 
	lies in $H^0(\Z[\tfrac{1}{N}],U/\nu U) \otimes_{\La_{\theta}} Q(\La_{\theta})$.
	We set $\mc{Q} = \La^{\iota}(1) \cozp Q(\La_{\theta})$ for brevity,
	and we let $V^*$ denote the $Q(\La_{\theta})$-dual of a $Q(\La_{\theta})\ps{G_{\Q}}$-module $V$.
	We then have
	\begin{align*}
		H^0(\Z[\tfrac{1}{N}],U/\nu U) \otimes_{\La_{\theta}} Q(\La_{\theta}) 
		&\cong \Hom_{Q(\La_{\theta})\ps{G_{\Q}}}(Q(\La_{\theta}),U/\nu U \otimes_{\La_{\theta}} Q(\La_{\theta}))\\
		&\cong \Hom_{Q(\La_{\theta})\ps{G_{\Q}}}(Q(\La_{\theta}),
		\T \otimes_{\La_{\theta}} \mc{Q}/\nu\mc{Q})\\
		&\cong \Hom_{Q(\La_{\theta})\ps{G_{\Q}}}((\T  \otimes_{\La_{\theta}} Q(\La_{\theta}))^*,
		\mc{Q}/\nu\mc{Q}),
	\end{align*}
	noting that $\T  \otimes_{\La_{\theta}} Q(\La_{\theta})$ is a finite-dimensional $Q(\La_{\theta})$-vector space 
	in the final step.  It follows from the work of Hida \cite{hida} that the rank two
	$\mf{h} \otimes_{\La_{\theta}} Q(\La_{\theta})$-module 
	$\T \otimes_{\La_{\theta}} Q(\La_{\theta})$ is a direct sum of irreducible two-dimensional representations 
	over finite extensions of $Q(\La_{\theta})$
	that are attached to ordinary $\La_{\theta}$-adic eigenforms for the Hecke operators $T_{\ell}$ with 
	$\ell \nmid N$ and $U_p$.
	Since $G_{\Q}$ acts on $\mc{Q}/\nu\mc{Q}$ through
	its maximal abelian quotient,
	any $Q(\La_{\theta})\ps{G_{\Q}}$-homomorphism 
	$$
		(\T \otimes_{\La_{\theta}} Q(\La_{\theta}))^* \to \mc{Q}/\nu\mc{Q}
	$$	 
	is trivial.
	Thus, the $\La \cozp \La_{\theta}$-module $H^1_{\Iw}(\Z_{\infty}[\tfrac{1}{N}],\T(1))$ has no 
	torsion.  
	
	It remains only to show that the rightmost group in the exact sequence
	$$
		0 \to H^1_{\Iw}(\Z_{\infty}[\tfrac{1}{N}],\T(1)) 
		\to H^1_{\Iw}(\Z_{\infty}[\tfrac{1}{N}],\tT(1))
		\to H^1_{\Iw}(\Z_{\infty}[\tfrac{1}{N}],(\tT/\T)(1)).
	$$
	has no $\La \cozp \La_{\theta}$-torsion.
	Taking $\bar{U} = \La^{\iota} \cozp (\tT/\T)(1)$,
	this reduces to showing that the $G_{\Q}$-invariant group $(\bar{U}/\nu\bar{U})^{G_{\Q}}$ is trivial 
	for $\nu$ as before.  
	But, $G_{\Q(\mu_M)}$ acts trivially
	on $\tT/\T$ by \cite[Prop.~3.2.4]{fk-proof}, so $(\bar{U}/\nu\bar{U})^{G_{\Q(\mu_M)}}$ 
	is killed by $\omega(a)-1$ for all $a \in (\Z/p\Z)^{\times}$, which of course means it is trivial.
\end{proof}

We next prove our first result on the integrality of the zeta maps of the first two authors.
We consider the following symbols in the completed tensor product $\La \cozp \tilde{\mc{H}}$:
\begin{equation} \label{fancycd}
	{}_{c,d}(u:v] = c^2d^2 \otimes (u:v) - d^2\kappa(c) \otimes (cu:v) - c^2\kappa(d) \otimes (u:dv) + \kappa(cd)
	\otimes (cu:dv),
\end{equation}
where $\kappa \colon \zp^{\times} \to \La$ sends a unit to the group element of its projection to $1+p\zp$ (see \cite[2.4.6]{fk-proof}).  

\begin{theorem} \label{zinfty}
	If $f = N$ (resp., $M \mid f$ and $f > 1$), then
	there exists a unique $\La \cozp \mf{H}$-linear map
	$$
		 z \colon \La \cozp  \tilde{\mc{H}} \to 
		 H^1_{\Iw}(\Z_{\infty}[\tfrac{1}{N}],\tT^{(\theta)}(1)) \quad
		\text{(resp., }z \colon \La \cozp  \tilde{\mc{H}} \to 
		 H^1_{\Iw}(\Z_{\infty}[\tfrac{1}{N}],\tT_{\mf{m}}(1)) \text{)} 
	$$
	such that $z({}_{c,d}(u:v]) = {}_{c,d} z(u,v)_{\theta}$ for all $c$, $d$, $u$, and $v$.
\end{theorem}	

\begin{proof}
	We let $\Omega = \La \cozp \La_{\theta}$.
	In \cite[Theorem 3.2.3]{fk-proof}, it is shown (via a regulator computation) 
	that there exists a well-defined map  
	$$
		z \colon  \La \cozp \tilde{\mc{H}} \to H^1_{\Iw}(\Z_{\infty}[\tfrac{1}{N}],\tT(1)) 
		\otimes_{\Omega} Q(\Omega)
	$$
	with the stated property.  Each ${}_{c,d} z(u:v)_{\theta}$ lies in
	$H^1_{\Iw}(\Z_{\infty}[\tfrac{1}{N}],\tT(1))$, and
	$\La \cozp \tilde{\mc{H}}$ is generated by the ${}_{c,d}(u:v]$ as an 
	$\Omega$-module by Propositions \ref{ordcoh} and \ref{spancoh} and Nakayama's lemma.
	Therefore, the image of $z$ is contained in the image of the map
	$$
		H^1_{\Iw}(\Z_{\infty}[\tfrac{1}{N}],\tT(1)) \to H^1_{\Iw}(\Z_{\infty}[\tfrac{1}{N}],\tT(1)) 
		\otimes_{\Omega} Q(\Omega)
	$$
	Lemma \ref{Iwtorsfree} tells us that the latter map is injective, hence the result.
\end{proof}

\begin{lemma}\ \label{torsfree}
	If $p \mid f$, then the $\La_{\theta}$-module 
	$H^1(\Z[\tfrac{1}{N}],\tT^{(\theta)}(1))$ is torsion-free.
	If $M \mid f$ and $f > 1$, 
	then $H^1(\Z[\frac{1}{N}],\tT_{\mf{m}}(1))$ is $\La_{\theta}$-torsion free.
\end{lemma}

\begin{proof}
	The first statement is \cite[Prop.~3.3.6(2)]{fk-proof}, which implies the second statement for $f = N$. 
	The second statement is proven (in particular) in the case $f = M > 1$ in \cite{sh-extn}.
\end{proof}

Let $\Cor \colon H^1_{\Iw}(\Z_{\infty}[\tfrac{1}{N}],\tT^{(\theta)}(1))
\to H^1(\Z[\tfrac{1}{N}],\tT^{(\theta)}(1))$ denote corestriction.  Then $\Cor \circ z$ factors through the quotient 
$\tilde{\mc{H}}$ of $\La \cozp \tilde{\mc{H}}$.

\begin{theorem} \label{zsharp}
	If $f = N$ (resp., $M \mid f$ and $f > 1$), then there exists a unique 
	$\mf{H}$-linear map
	$$
		z^{\sharp} \colon \mc{H} \to H^1(\Z[\tfrac{1}{N}],\tT^{(\theta)}(1)) \quad \text{(resp., } 
		z^{\sharp} \colon \mc{H}
		\to H^1(\Z[\tfrac{1}{N}],\tT_{\mf{m}}(1)) \text{)}
	$$
	such that $(1-U_p)z^{\sharp} = \Cor \circ z$ on $\mc{H}$ and
	$$
		z^{\sharp}({}_{c,d}(u:v)) = {}_{c,d} z^{\sharp}(u:v)_{\theta} 
	$$
	for all $c$, $d$, $u$, and $v$.
\end{theorem}

\begin{proof}
	In \cite[Theorem 3.3.9]{fk-proof}, it was shown that there exists a map 
	$$
		z^{\sharp} \colon \mc{H} \to H^1(\Z[\tfrac{1}{N}],\tT(1)) \otimes_{\La_{\theta}} Q(\La_{\theta})
	$$	
	with the stated properties, equalities taking place in this group.  
	Moreover, the elements ${}_{c,d} z^{\sharp}(u:v)_{\theta}$ are all in
	$H^1(\Z[\tfrac{1}{N}],\tT(1))$ itself, and $\mc{H}$ is generated by the ${}_{c,d}(u:v)$ (with $N \nmid u$) as a 
	$\La_{\theta}$-module by Propositions \ref{ordcoh} and \ref{spancoh}.  Therefore, the image
	of $z^{\sharp}$ is contained in the image of the map
	$$
		 H^1(\Z[\tfrac{1}{N}],\tT(1)) \to H^1(\Z[\tfrac{1}{N}],\tT(1)) \otimes_{\La_{\theta}} Q(\La_{\theta}),
	$$
	which is injective by Lemma \ref{torsfree}.  The equalities then also take place in $H^1(\Z[\tfrac{1}{N}],\tT(1))$.
\end{proof}

\begin{remark}
	Let $X_1^0(N_r)_{/\Z[\frac{1}{N}]}$ denote the scheme-theoretic complement in $X_1(N_r)_{/\Z[\frac{1}{N}]}$
	of the closed $\Z[\frac{1}{N}]$-subscheme of zero cusps,
	and set
	$$
		\T^0 = \varprojlim_r H^1(X_1^0(N_r)_{/\overline{\Q}},\zp(1))^{\ord}.
	$$
	It follows from \cite[Prop.~3.2.4]{fk-proof} that the map
	$H^1(\Z[\tfrac{1}{N}],\T^0(1)) \to H^1(\Z[\tfrac{1}{N}],\tT(1))$
	is injective.  By \cite[Theorem 3.3.9(iii)]{fk-proof},
	the map $z^{\sharp}$ then actually takes values in (the $\theta$-eigenspace, or Eisenstein part, of)
	$H^1(\Z[\tfrac{1}{N}],\T^0(1))$ inside $H^1(\Z[\tfrac{1}{N}],\tT(1))$.
\end{remark}

\renewcommand{\baselinestretch}{1}

\end{document}